\pdfoutput=1

\documentclass[final]{siamltex}

\usepackage{amsfonts}
\usepackage[leqno]{amsmath}
\usepackage{ amssymb }
\usepackage{tikz}
\usepackage{algorithmic}
\usepackage{url}
\usepackage{multirow}
\usepackage{colortbl}
\usepackage{booktabs}
\usepackage{siunitx}
\usepackage{mathtools}
\usepackage{caption}
\usepackage{subcaption}
\usepackage{comment}
\newtheorem{algorithm}{Algorithm}[section]
\newtheorem{example}{Example}[section]

\newtheorem{condition}{Condition}

\newcommand{\ten}[1]{\mathcal{#1}}  

\title{A constructive arbitrary-degree Kronecker product decomposition of tensors}

\author{Kim Batselier \and Ngai Wong\thanks{Department of Electrical and Electronic Engineering, The University of Hong Kong, Hong Kong}}

\begin{document}

\maketitle

\begin{abstract}
We propose the tensor Kronecker product singular value decomposition~(TKPSVD) that decomposes a real $k$-way tensor $\ten{A}$ into a linear combination of tensor Kronecker products with an arbitrary number of $d$ factors $\ten{A} = \sum_{j=1}^R \sigma_j\, \ten{A}^{(d)}_j \otimes \cdots \otimes \ten{A}^{(1)}_j$. We generalize the matrix Kronecker product to tensors such that each factor $\ten{A}^{(i)}_j$ in the TKPSVD is a $k$-way tensor. The algorithm relies on reshaping and permuting the original tensor into a $d$-way tensor, after which a polyadic decomposition with orthogonal rank-1 terms is computed. We prove that for many different structured tensors, the Kronecker product factors $\ten{A}^{(1)}_j,\ldots,\ten{A}^{(d)}_j$ are guaranteed to inherit this structure. In addition, we introduce the new notion of general symmetric tensors, which includes many different structures such as symmetric, persymmetric, centrosymmetric, Toeplitz and Hankel tensors. 
\end{abstract}

\begin{keywords} 
Kronecker product, structured tensors, tensor decomposition, TTr1SVD, generalized symmetric tensors, Toeplitz tensor, Hankel tensor
\end{keywords}

\begin{AMS}
15A69, 15B05, 15A18, 15A23, 15B57
\end{AMS}

\pagestyle{myheadings}
\thispagestyle{plain}
\markboth{KIM BATSELIER, NGAI WONG}{TKPSVD}

\section{Introduction}
Consider the singular value decomposition (SVD) of the following $16 \times 9$ matrix $\tilde{A}$
\begin{equation}
\footnotesize{
 \begin{pmatrix}
   1.108&  -0.267&  -1.192&  -0.267&  -1.192&  -1.281&  -1.192&  -1.281&   1.102\\
   0.417&  -1.487&  -0.004&  -1.487&  -0.004&  -1.418&  -0.004&  -1.418&  -0.228\\
  -0.127&   1.100&  -1.461&   1.100&  -1.461&   0.729&  -1.461&   0.729&   0.940\\
  -0.748&  -0.243&   0.387&  -0.243&   0.387&  -1.241&   0.387&  -1.241&  -1.853\\
   0.417&  -1.487&  -0.004&  -1.487&  -0.004&  -1.418&  -0.004&  -1.418&  -0.228\\
  -0.127&   1.100&  -1.461&   1.100&  -1.461&   0.729&  -1.461&   0.729&   0.940\\
  -0.748&  -0.243&   0.387&  -0.243&   0.387&  -1.241&   0.387&  -1.241&  -1.853\\
  -0.267&  -1.192&  -1.281&  -1.192&  -1.281&   1.102&  -1.281&   1.102&  -0.474\\
  -0.127&   1.100&  -1.461&   1.100&  -1.461&   0.729&  -1.461&   0.729&   0.940\\
  -0.748&  -0.243&   0.387&  -0.243&   0.387&  -1.241&   0.387&  -1.241&  -1.853\\
  -0.267&  -1.192&  -1.281&  -1.192&  -1.281&   1.102&  -1.281&   1.102&  -0.474\\
  -1.487&  -0.004&  -1.418&  -0.004&  -1.418&  -0.228&  -1.418&  -0.228&  -1.031\\
  -0.748&  -0.243&   0.387&  -0.243&   0.387&  -1.241&   0.387&  -1.241&  -1.853\\
  -0.267&  -1.192&  -1.281&  -1.192&  -1.281&   1.102&  -1.281&   1.102&  -0.474\\
  -1.487&  -0.004&  -1.418&  -0.004&  -1.418&  -0.228&  -1.418&  -0.228&  -1.031\\
   1.100&  -1.461&   0.729&  -1.461&   0.729&   0.940&   0.729&   0.940&  -0.380
 \end{pmatrix}}
 \label{eqn:exampleA}
\end{equation}
into a $16 \times 16$ orthogonal matrix $U$, a $16 \times 9$ diagonal matrix $S$ and a $9\times 9$ orthogonal matrix $V$. The distinct entries of $\tilde{A}$ were drawn from a standard normal distribution. Note that $\tilde{A}$ has only 5 distinct columns and 7 distinct rows, limiting the rank to maximally 5, and has no apparent structure such as a symmetry or displacement structure (Hankel or Toeplitz matrix). Reshaping the first left and right singular vectors into a $4\times 4$ and $3 \times 3$ matrix respectively results in the following Hankel matrices
$$
\begin{pmatrix}
  -0.100&   0.194&  -0.375&   0.245\\
   0.194&  -0.375&   0.245&  -0.244\\
  -0.375&   0.245&  -0.244&  -0.177\\
   0.245&  -0.244&  -0.177&   0.106
\end{pmatrix},
\begin{pmatrix}
   0.036&  -0.158&   0.442\\
  -0.158&   0.442&  -0.373\\
   0.442&  -0.373&  -0.290
\end{pmatrix}.
$$
In fact, it turns out that all singular vectors of $\tilde{A}$ can be reshaped into square Hankel matrices. In the same vein, other matrices can be easily constructed such that their reshaped singular vectors are square structured matrices (symmetric, Toeplitz, centrosymmetric, persymmetric,...). And as we will demonstrate in this article, the same can even be done for $k$-way tensors. Now what is going on here? Why are all singular vectors so highly structured? This is explained in this article, where we introduce the tensor-based Kronecker product (KP) singular value decomposition (TKPSVD). The TKPSVD decomposes an arbitrary real $k$-way tensor $\ten{A} \in \mathbb{R}^{n_1 \times n_2 \times \cdots \times n_k}$ as
\begin{equation}
\ten{A} \;=\; \sum_{j=1}^R \sigma_j\, \ten{A}^{(d)}_j \otimes \cdots \otimes \ten{A}^{(1)}_j,
\label{eq:tkpsvd}
\end{equation}
where $\otimes$ denotes the tensor Kronecker product, defined in Section \ref{sec:Kronecker}, and the tensors $\ten{A}^{(i)}_j \in \mathbb{R}^{n^{(i)}_1 \times \cdots \times n^{(i)}_k}$ satisfy
\begin{align}
||\ten{A}^{(i)}_j ||_F \;=\; 1 \quad \textrm{with} \quad 
\prod_{i=1}^d n^{(i)}_r=n_r \quad (r \in \{1,\ldots,k\}).
\label{eq:dimconstraint}
\end{align}
It turns out that the TKPSVD is key in explaining the structured singular vectors of the matrix $\tilde{A}$. In Section \ref{sec:tkpsvd}, we show that computing the SVD of \eqref{eqn:exampleA} is an intermediate step in the computation of the TKPSVD of a $12 \times 12$ Hankel matrix into the Kronecker product of a $4\times 4$ with a $3\times 3$ matrix.\\
\\
The Kronecker rank \cite{Hackbusch03hierarchicalkronecker}, which is generally hard to compute, is defined as the minimal $R$ required in \eqref{eq:tkpsvd} in order for the equality to hold. If for any $r \in \{1,\ldots,k\}$ we have that $n\triangleq n^{(1)}_r=n^{(2)}_r=\cdots=n^{(d)}_r$, then $n_r=n^d$. For this reason, we call the number of factors $d$ in \eqref{eq:tkpsvd} the degree of the decomposition. The user of the TKPSVD algorithm is completely free to choose the degree $d$ and the dimensions $n^{(i)}_r$ of each of the KP factors, as long as they satisfy \eqref{eq:dimconstraint}. In addition to the development of the TKPSVD algorithm, another major contribution of this article, shown for the first time in the literature, is the proof that by our proposed algorithm we will have the following favorable structure-preserving properties when all $\ten{A}^{(i)}_j$ and $\ten{A}$ are cubical:
$$
\textrm{if $\ten{A}$ is } \left\{
\begin{array}{l}
 \textrm{symmetric}\\
 \textrm{persymmetric}\\
\textrm{centrosymmetric}\\ 
 \textrm{Toeplitz}\\
 \textrm{Hankel}\\
 \textrm{general symmetric}
\end{array}\right\}
\textrm{then each } \ten{A}^{(i)}_j \textrm{ is}
\left\{
\begin{array}{l}
 \textrm{(skew-)symmetric}\\
 \textrm{(skew-)persymmetric}\\
 \textrm{(skew-)centrosymmetric}\\ 
 \textrm{Toeplitz}\\
 \textrm{Hankel}\\
 \textrm{general (skew-)symmetric}\\
\end{array}\right\}.
$$
The fact that each of the factors $\ten{A}^{(i)}_j$ inherits the structure of $\ten{A}$ is not trivial. In providing this proof a very natural generalization of symmetric tensors is introduced, which we name \textit{general symmetric tensors}. In addition, Toeplitz and Hankel tensors are also generalized into what we call \textit{shifted-index structures}, which are special cases of general symmetries. In fact, when $\ten{A}$ is general symmetric, then all its cubical factors $\ten{A}^{(i)}_j$ will also be general symmetric. Another interesting feature of the TKPSVD algorithm is that if the summation in \eqref{eq:tkpsvd} is limited to the first $r$ terms, then the relative approximation error in the Frobenius norm is given by
\begin{equation}
\frac{||A-\sum_{j=1}^r \sigma_j\, A^{(d)}_j \otimes \cdots \otimes A^{(1)}_j ||_F } {||A||_F} = \frac{\sqrt{\sigma_{r+1}^2 + \cdots + \sigma_R^2}} {\sqrt{\sigma_{1}^2 + \cdots + \sigma_R^2}}.
\label{eq:approxerror}
\end{equation}
Equation \eqref{eq:approxerror} has the computational advantage that the relative approximation error can be easily obtained from the $\sigma_j$'s without having to explicitly construct the approximant. 

\subsection{Context}
The TKPSVD is directly inspired by the work of Van Loan and Pitsianis~\cite{Loan93approximationwith}. In their paper, they solve the problem of finding matrices $B,C$ such that $||A-B\otimes C||_F$ is minimized. The globally minimizing matrices $B,C$ turn out to be the singular vectors corresponding with the largest singular value of a particular permutation of $A$. In \cite{Loan200085}, the full SVD of the permuted $A$ is considered and the corresponding decomposition of $A$ into a linear combination of Kronecker products is called the Kronecker product singular value decomposition (KPSVD). Applications of the $d=2,k=2$ KPSVD approximation in image restoration are described in \cite{Nagy2000,Perrone2003}, whereas extensions to the $d=3,k=2$ case using the higher order singular value decomposition (HOSVD), also for imaging, are described in \cite{Kilmer2006,Elden2014}.

The decomposition in this paper is a direct generalization of the KPSVD to an arbitrary number of KP factors $d$ and to arbitrary $k$-way tensors. In the TKPSVD case, the SVD is replaced by either a canonical polyadic decomposition (CPD)~\cite{candecomp,harshman1970fpp} with orthogonal factor matrices, the HOSVD~\cite{Lathauwer:HOSVD} or the tensor-train rank-1 SVD (TTr1SVD)~\cite{ttr1svd15}. The TKPSVD reduces to the KPSVD for the case $d=2$ and $k=2$. Contrary to previous work in the literature, we are not interested in minimizing $|| \ten{A}-\sum_{j=1}^r \ten{A}^{(d)}_j\otimes  \cdots \otimes  \ten{A}^{(1)}_j||_F$. Instead, we are interested in a full decomposition such that any structure of $\ten{A}$ is also present in the KP factors. Explicit decomposition algorithms for when $d\geq 4$ and $k \geq 3$ are not found in the literature. To this end, our proposed TKPSVD algorithm readily works for any degree $d$ and any tensor order $k$, does not require any a priori knowledge of the number of KP terms, and preserves general symmetry in the KP factors $\ten{A}^{(i)}_j$. Furthermore, a Matlab/Octave implementation that uses the TTr1SVD and works for any arbitrary degree $d$ and order $k$ can be freely downloaded from \url{https://github.com/kbatseli/TKPSVD}. In brief, the contributions of this article are:
\begin{itemize}
\item an explicit formulation of the generalization of the KPSVD algorithm for $d>2$ and $k > 2$ is presented for the first time in the literature,
\item a new notion of general symmetric tensors, which describes many tensor structures, is introduced,
\item we prove that for a general symmetric tensor $\ten{A}$ all KP factors $\ten{A}^{(i)}_j$ in \eqref{eq:tkpsvd} are guaranteed to have the same general symmetry under a particular condition.
\end{itemize}

The outline of this article is as follows. In Section \ref{sec:notations}, we introduce some basic tensor concepts and notations. In Section \ref{sec:Kronecker}, we generalize the matrix Kronecker product into the tensor Kronecker product and present some of its properties. In Section \ref{sec:tkpsvd}, we first derive the theorem that underlies the TKPSVD and then present the TKPSVD algorithm for both general and diagonal tensors. In Section \ref{sec:gensym}, we introduce the framework of \textit{general symmetry} that describes many different structured tensors such as symmetry, centrosymmetry, Hankel, Toeplitz, etc.. Preservation of general symmetry in the KP factors when the original tensor is general symmetric is proven in Section \ref{sec:strucproofs}. Numerical experiments that demonstrate different aspects of the TKPSVD are presented in Section \ref{sec:applications}, after which conclusions follow.

\section{Tensor basics and notation}
\label{sec:notations}
We only consider real matrices and tensors in this paper. Scalars are denoted by greek letters $(\alpha, \beta,...)$, vectors by lowercase letters $(a,b,...)$, matrices by uppercase letters $(A,B,...)$ and higher-order tensors by uppercase caligraphic letters $(\ten{A},\ten{B},...)$. The notation $(\cdot)^T$ denotes the transpose of either a vector or a matrix. A $k$th-order or $k$-way tensor is a multi-way array $\ten{A} \in \mathbb{R}^{n_1\times n_2 \times \cdots \times n_k}$. A tensor is cubical if all dimensions are equal, e.g., $n_1=n_2=\cdots=n_k$. 

Entries of tensors are always denoted with square brackets around the indices. This enables an easy way of representing the grouping of indices. Suppose for example that $\ten{A}$ is a 4-way tensor with entries $\ten{A}_{[i_1][i_2][i_3][i_4]}$. To improve readability we do not write the square brackets when all indices are considered separate, therefore $\ten{A}_{i_1i_2i_3i_4}\triangleq \ten{A}_{[i_1][i_2][i_3][i_4]} $. A $3$-way tensor can now be formed by grouping, for example, the first two indices together. The entries of this $3$-way tensor are then denoted by $\ten{A}_{[i_1i_2][i_3][i_4]}$, where the grouped index $[i_1i_2]$ is easily converted into the linear index $i_1+n_1(i_2-1)$. Grouping the indices into the $[i_1]$ and $[i_2i_3i_4]$ results in a $n_1 \times n_2n_3n_4$ matrix with entries $\ten{A}_{[i_1][i_2i_3i_4]}$. The column index $[i_2i_3i_4]$ is equivalent to the linear index $i_2+n_2 (i_3-1)+n_2 n_3 (i_4-1)$. A very special case of grouping indices is obtained when all indices are grouped together. The resulting vector is then called the vectorization of $\ten{A}$, denoted $\textrm{vec}(\ten{A})$, with entries $\ten{A}_{[i_1i_2i_3i_4]}$.

The $r$-mode product of a tensor $\ten{A} \in \mathbb{R}^{n_1\times n_2 \times \cdots \times n_k}$ with a matrix $U \in \mathbb{R}^{p \times n_r}$ is defined by
$$
(\ten{A} \times_r U)_{i_1\cdots i_{k-1} j i_{k+1} \cdots i_d} \;=\; \sum_{i_r=1}^{n_r} U_{j i_r} \ten{A}_{i_1 \cdots i_r\cdots i_d},
$$
so that $\ten{A}{\times_r} U \in \mathbb{R}^{n_1 \times \cdots \times n_{r-1} \times p \times n_{r+1} \times \cdots \times n_d}$. The multiplication of a $k$-way tensor $\ten{A}$ along all its modes with matrices $P_1,\ldots,P_k$
$$
\ten{B} \;=\; \ten{A} \times_1 P_1 \times_2 P_2 \times_3 \cdots {\times_k} P_k,
$$
can be rewritten as the following linear system
\begin{equation}
\textrm{vec}(\ten{B}) \;=\; (P_d \otimes \cdots \otimes P_2 \otimes P_1)\, \textrm{vec}(\ten{A}),
\label{eq:modeprodvec}
\end{equation}
where $\otimes$ is the conventional matrix Kronecker product~\cite{regalia1989}, which is defined and generalized to the tensor case in Section \ref{sec:Kronecker}. The inner product between two tensors $\ten{A},\ten{B} \in \mathbb{R}^{n_1\times \cdots \times n_d}$ is defined as
$$
\langle \ten{A},\ten{B} \rangle \;=\; \sum_{i_1,i_2,\cdots,i_d}\,\ten{A}_{i_1i_2 \cdots i_d}\,\ten{B}_{i_1i_2 \cdots i_d} \;=\; \textrm{vec}(\ten{A})^T \, \textrm{vec}(\ten{B}).
$$
Two tensors $\ten{A},\ten{B}$ are orthogonal with respect to one another when $\langle \ten{A},\ten{B} \rangle=0$. The norm of a tensor is taken to be the Frobenius norm $||\ten{A}||_F=\langle \ten{A},\ten{A} \rangle^{1/2}$. A $k$-way rank-1 tensor $\ten{A}\in \mathbb{R}^{n_1 \times \cdots \times n_k}$ is per definition the outer product~\cite{tensorreview}, denoted by $\circ$, of $k$ vectors $a^{(i)} \in \mathbb{R}^{n_i} \, (i \in \{1,\ldots,k\})$ such that
\begin{equation}
\ten{A}=a^{(1)} \circ a^{(2)} \circ \cdots \circ a^{(k)} \quad \textrm{with} \quad \ten{A}_{i_1 i_2 \cdots i_k}\;=\; a^{(1)}_{i_1}\,a^{(2)}_{i_2}\,\cdots \,a^{(k)}_{i_k}.
\label{def:outerprod}
\end{equation}
Any real $k$-way tensor $\ten{A}$ can always be written as a linear combination of rank-1 terms
\begin{equation}
\ten{A} \;=\; \sum_{j=1}^R \sigma_j\, a^{(1)}_{j}\circ a^{(2)}_{j}\circ \cdots \circ a^{(k)}_{j},
\label{eq:cpd}
\end{equation}
where $\sigma_j \in \mathbb{R}$ and all the $a^{(i)}_j$ vectors satisfy $||a^{(i)}_j||_2=1$. Such a decomposition is called a polyadic decomposition (PD) of the tensor $\ten{A}$. When the equality in \eqref{eq:cpd} holds with a minimal number of terms $R$, then the PD is called canonical (CPD)~\cite{candecomp,harshman1970fpp}. The number $R$ in the CPD is called the \textit{tensor rank}. Unlike the SVD, each of the rank-1 terms in the (C)PD is not necessarily orthogonal.

The Tucker decomposition~\cite{tuckerreview,tucker64extension} writes a $k$-way tensor ${\ten{A}}$ as the following multilinear transformation of a core tensor $\ten{S} \in \mathbb{R}^{r_1 \times r_2 \times \cdots \times r_k}$ by factor matrices $U^{(i)} \in \mathbb{R}^{n_i \times r_i}$
\begin{equation}
{\ten{A}}\;=\; \ten{S} \times_1 U^{(1)} \times_2 U^{(2)} \times_3 \cdots \times_k U^{(k)},
\label{eq:tucker}
\end{equation}
which can also be written as \eqref{eq:cpd} where each $\sigma_j$ is now an entry of the core tensor $\ten{S}$. Each rank-1 term of the Tucker decomposition is then given by
$$
\ten{S}_{i_1i_2\cdots i_k} \; U^{(1)}(:,i_1) \circ U^{(2)}(:,i_2) \circ \cdots \circ U^{(k)}(:,i_k),
$$
where we use Matlab colon notation to denote columns of the $U^{(i)}$ factor matrices. The minimal size of the core tensor $\ten{S}$ such that the equality in \eqref{eq:tucker} holds is called the \textit{multilinear rank}. The higher-order SVD (HOSVD)~\cite{Lathauwer:HOSVD} is a Tucker decomposition where the core tensor $\ten{S}$ has the same dimensions as the original tensor $\ten{A}$ and with the additional property that the factor matrices $U^{(i)}$ and the slices of $\ten{S}$ in the same mode are orthogonal. This implies that each rank-1 term is orthogonal to all other rank-1 terms in the HOSVD, which has the immediate advantage that the approximation error can be determined as in \eqref{eq:approxerror}.

The PARATREE/TTr1SVD decomposition~\cite{ttr1svd15,Koivunen2008} is another decomposition of a $k$-way tensor into orthogonal rank-1 terms as described by \eqref{eq:cpd}. The total number of terms in the TTr1SVD is upperbounded by $R= \prod_{r=0}^{k-2} \textrm{min}(n_{r+1},\prod_{i=r+2}^k n_i)$ and therefore depends on the ordering of the indices. This decomposition is computed from repeated reshapings and SVD computations and is unique for a fixed order of indices. Note that although each of the rank-1 terms is orthogonal with respect to all others, unlike the HOSVD, the factor matrices $U^{(i)}$ are not orthogonal. In addition, the scalar $\sigma_j$ coefficients obtained in the PARATREE/TTR1SVD decomposition are guaranteed to be nonnegative. This has the advantage that one can plot these $\sigma_j$ coefficients in descending order and inspect the relative weight of each of the rank-1 terms in a very straightforward manner, just like one can do with the singular values of a matrix.

\section{Tensor Kronecker product}
\label{sec:Kronecker}
\subsection{Definition}
The definition of the Kronecker product for two matrices is well-known. If $B \in \mathbb{R}^{m_1 \times m_2}$ and $C \in \mathbb{R}^{n_1 \times n_2}$, then their Kronecker product $B \otimes C$ is an $m_1 \times m_2$ block matrix whose $(i_3,i_4)$th block is the $n_1 \times n_2$ matrix $B_{i_3i_4}C$
\begin{equation}
\begin{pmatrix}
B_{11} & \cdots & B_{1n_1}\\
\vdots & \ddots & \vdots \\
B_{m_11} & \cdots &B_{m_1n_1}\\
\end{pmatrix} \otimes C \;=\;
\begin{pmatrix}
B_{11}C & \cdots & B_{1n_1}C\\
\vdots & \ddots & \vdots \\
B_{m_11}C & \cdots & B_{m_1n_1}C\\
\end{pmatrix}.
\label{def:kron}
\end{equation}
Generalizing this definition to the Kronecker product of $k$-way tensors is quite straightforward, although not in the form as given by \eqref{def:kron}. Before giving the definition of the tensor Kronecker product, we first investigate how the entries of the matrix Kronecker product are described by the indices of the original matrices. The following lemma is easily verified.\\
\begin{lemma}
\label{lemma:kronentry}
If the entries of the matrices $B \in \mathbb{R}^{m_1 \times m_2},C\in \mathbb{R}^{n_1 \times n_2}$ are denoted $B_{i_3i_4}$ and $C_{i_1i_2}$ respectively, then the entries of their Kronecker product $A = B \otimes C$ are described by $A_{[i_1i_3][i_2i_4]} = B_{i_3i_4}\,C_{i_1i_2}$ for all possible values of $i_1,i_2,i_3,i_4$. 
\end{lemma}

Remember that the grouped indices $[i_1i_3],[i_2i_4]$ are easily converted into the linear row index $i_1+n_1(i_3-1)$ and the linear column index $i_2+n_2(i_4-1)$ respectively. The definition of the tensor Kronecker product follows from the generalization of Lemma \ref{lemma:kronentry} to multiple indices.
\begin{definition}
\label{def:tkron}
Let $\ten{B} \in \mathbb{R}^{n_1 \times n_2 \times \cdots \times n_k},\ten{C} \in \mathbb{R}^{m_1 \times m_2 \times \cdots \times m_k}$ be two $k$-way tensors with entries denoted by $\ten{B}_{i_{k+1}\cdots i_{2k}},\ten{C}_{i_1\cdots i_k}$ respectively. The tensor Kronecker product $\ten{A} = \ten{B} \otimes \ten{C} \in  \mathbb{R}^{n_1m_1 \times n_2m_2 \times \cdots \times n_km_k}$ is then defined from 
$$
\ten{A}_{[i_1i_{k+1}][i_2i_{k+2}]\cdots [i_ki_{2k}]} = \ten{B}_{i_{k+1}\cdots i_{2k}}\,\ten{C}_{i_1\cdots i_k},
$$
which needs to hold for all possible values of $i_1,\ldots,i_{2k}$.
\end{definition}

Although the tensor Kronecker product is a very straightforward generalization of the matrix Kronecker product, we failed to find any reference to it in the literature. One possible implementation of the tensor Kronecker product would be to use Definition \ref{def:tkron} over all possible values of the indices $i_1,\ldots,i_{2k}$ but this would not be very efficient. Instead, one can use an existing implementation of the vector/matrix Kronecker product (`kron.m' in Matlab) on the vectorized tensors $\textrm{vec}(\ten{B}),\textrm{vec}(\ten{C})$. Indeed, the entries of $c\triangleq \textrm{vec}(\ten{B}) \otimes \textrm{vec}(\ten{C})$ are indexed by the single grouped index $[i_1 \cdots i_k i_{k+1}\cdots i_{2k}]$. One can then reshape and permute the entries of $c$ such that the desired $[i_1i_{k+1}][i_2i_{k+2}]\cdots [i_ki_{2k}]$ index structure is obtained. This is how the tensor Kronecker product is implemented in our Matlab/Octave TKPSVD package.

\subsection{Properties of the tensor Kronecker product}
We briefly list some properties of the tensor Kronecker product without going into details. The following properties are easily verified
\begin{align*}
\ten{A} \otimes (\ten{B} + \ten{C}) &= \ten{A} \otimes \ten{B} + \ten{A} \otimes \ten{B}, \\
(\ten{A} + B) \otimes \ten{C} &= \ten{A} \otimes \ten{C} + \ten{B} \otimes \ten{C}, \\
(\alpha\ten{A}) \otimes \ten{B} &= \ten{A} \otimes (\alpha\ten{B}) = \alpha(\ten{A}\otimes \ten{B}),\\
(\ten{A} \otimes \ten{B}) \otimes \ten{C} &= \ten{A} \otimes (\ten{B} \otimes \ten{C}),
\end{align*}
where $\ten{A},\ten{B},\ten{C}$ are $k$-way tensors and $\alpha$ is a scalar. Just as in the matrix case, the tensor Kronecker product is not commutative but permutation equivalent. That is, there exists permutation matrices $P_1,\ldots,P_k$ such that
\begin{equation}
\ten{A} \otimes \ten{B} \;=\; (\ten{B} \otimes \ten{A}) \, \times_1 \, P_1 \, \times_2 \, \cdots \, \times_k P_k.
\label{eq:permequiv}
\end{equation}
This is easily seen from the definition. Indeed, suppose we have that $\ten{C}= \ten{A} \otimes \ten{B}$ and $\tilde{\ten{C}} = \ten{B} \otimes \ten{A}$, then
\begin{align*}
\ten{C}_{[i_1i_{k+1}][i_2i_{k+2}]\cdots [i_ki_{2k}]} &= \ten{A}_{i_{k+1}\cdots i_{2k}} \, \ten{B}_{i_1 \cdots i_k},\\
\tilde{\ten{C}}_{[i_{k+1}i_1][i_{k+2}i_2]\cdots [i_{2k}i_k]} &= \ten{A}_{i_1\cdots i_k} \, \ten{B}_{i_{k+1}\cdots i_{2k}}.
\end{align*}
Now let $P_1,\ldots,P_k$ be the permutation matrices that swap $[i_{k+1}i_1]$ into $[i_1i_{k+1}]$, $[i_{k+2}i_2]$ into $[i_2i_{k+2}]$, \ldots, $[i_{2k}i_k]$ into $[i_ki_{2k}]$ respectively, then \eqref{eq:permequiv} follows. Furthermore, if $\ten{A},\ten{B}$ are cubical and of the same dimension then $\ten{C}$ and $\tilde{\ten{C}}$ are permutation similar, which means that $P_1=P_2=\cdots=P_k$.

The mixed-product property of the Kronecker product states that if $A,B,C,D$ are matrices such that one can form the matrix products $AC,BD$, then $(A \otimes B)(C\otimes D)=(AC) \otimes (BD)$. This is called the mixed-product property, because it mixes the ordinary matrix product and the Kronecker product. We can also write the mixed-product property using the $1$-mode product as $(C \otimes D)\times_1(A \otimes B)=(C\times_1 A) \otimes (D\times_1 B)$. Its generalization involves the mixing of the $r$-mode product with the tensor Kronecker product. Let $A,B$ be matrices and $\ten{C},\ten{D}$ be $k$-way tensors of appropriate dimensions then for any $r \in \{1,\ldots,k\}$
\begin{equation}
(\ten{C} \otimes \ten{D}) \, \times_r \, (A \otimes B) = (\ten{C} \, \times_r \, A) \otimes (\ten{D} \, \times_r \, B).
\label{eq:prodmix}
\end{equation}
What the mixed-product property tells us is that we can obtain the $r$-mode product of the tensor $\ten{C} \otimes \ten{D}$ with the matrix $A \otimes B$ from the Kronecker product of $(\ten{C}\times_r A)$ with $(\ten{D} \times_r B)$. Choosing $r=1$ and replacing $\ten{C},\ten{D}$ with matrices in \eqref{eq:prodmix} results in the matrix mixed-product property.

\section{TKPSVD Theorem and algorithm}
\label{sec:tkpsvd}
Using Definition \ref{def:tkron} we can easily extend the tensor Kronecker product to multiple factors. Suppose we have three $3$-way tensors $\ten{A}^{(1)},\ten{A}^{(2)},\ten{A}^{(3)}$, then their Kronecker product $\ten{A} =  \ten{A}^{(3)} \otimes \ten{A}^{(2)} \otimes \ten{A}^{(1)}$
is completely characterized by the following relationship
\begin{equation}
\ten{A}_{[i_1i_4i_7][i_2i_5i_8][i_3i_6i_9]} =  \ten{A}^{(1)}_{i_1i_2i_3}\, \ten{A}^{(2)}_{i_4i_5i_6}\,\ten{A}^{(3)}_{i_7i_8i_9}.
\label{eq:tkronentries}
\end{equation}
Now suppose we permute all entries such that $\ten{\tilde{A}}_{[i_1i_2i_3][i_4i_5i_6][i_7i_8i_9]} \triangleq \ten{A}_{[i_1i_4i_7][i_2i_5i_8][i_3i_6i_9]}$ and that this is a rank-1 tensor. According to \eqref{def:outerprod}, $\ten{\tilde{A}}$ can then be written as the following outer product of vectors $\ten{\tilde{A}}= \,a^{(1)} \circ a^{(2)} \circ a^{(3)}$ with
\begin{equation}
\ten{\tilde{A}}_{[i_1i_2i_3][i_4i_5i_6][i_7i_8i_9]} =  \, a^{(1)}_{[i_1i_2i_3]}\, a^{(2)}_{[i_4i_5i_6]}\,a^{(3)}_{[i_7i_8i_9]}.
\label{eq:outerprodentries}
\end{equation}
Comparison of \eqref{eq:tkronentries} with \eqref{eq:outerprodentries} allows us to conclude that $a^{(1)} = \textrm{vec}(\ten{A}^{(1)}), a^{(2)} = \textrm{vec}(\ten{A}^{(2)})$ and $a^{(3)} = \textrm{vec}(\ten{A}^{(3})$. We formalize this observation in the following Theorem.
\begin{theorem}
For a given $k$-way tensor $\ten{A} \in \mathbb{R}^{n_1 \times n_2 \times \cdots \times n_k}$, if
\begin{equation}
\ten{A} \;=\; \sum_{j=1}^R \sigma_j \, \ten{A}^{(d)}_j \otimes  \cdots \otimes \ten{A}^{(2)}_j \otimes \ten{A}^{(1)}_j
\; \textrm{ and }\;
\ten{\tilde{A}} \;=\;\sum_{j=1}^R \sigma_j \, a^{(1)}_j \circ a^{(2)}_j \circ \cdots \circ a^{(d)}_j,
\label{eq:outerkron}
\end{equation}
where $\ten{\tilde{A}}$ is the permutation of $\ten{A}$ such that the indices of the $a^{(i)}_j$ vectors are identical to those of the $k$-way $\ten{A}^{(i)}_j$ tensors, then $a^{(i)}_j = \textrm{vec}(\ten{A}^{(i)}_j)$ for all $i\in \{1,\ldots,d\}, j \in \{1,\ldots,R\}$.
\label{theo:TKPSVD}
\end{theorem}

\begin{figure}[ht]
\centering
\includegraphics[width=.95\textwidth]{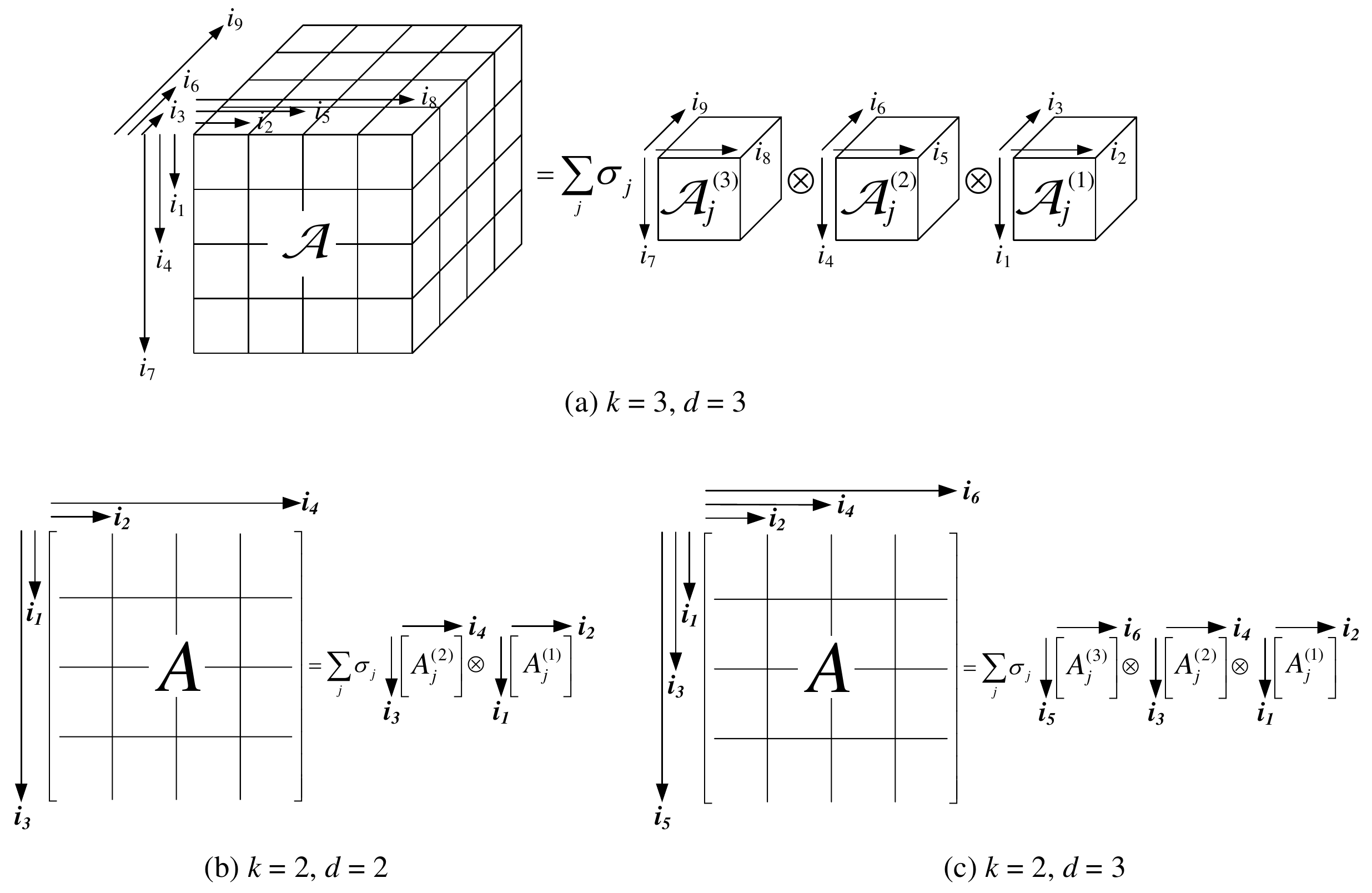}
\caption{How the grouped indices of $\ten{A}$ relate to the indices of the KP factors $\ten{A}^{i}_j$.}
\label{fig:d3kp}
\end{figure}

Observe that the order of the Kronecker products in \eqref{eq:outerkron} is reversed with respect to the order of the outer products. Theorem \ref{theo:TKPSVD} is crucial for the TKPSVD algorithm, since it tells us that the desired decomposition \eqref{eq:tkpsvd} can be computed from a PD of $\ten{\tilde{A}}$. We now derive the TKPSVD algorithm by means of a simple example. Suppose we have a $3$-way tensor $\ten{A}$
for which we want to find a degree-3 decomposition. This implies, as shown in \eqref{eq:tkronentries}, that each entry of $\ten{A}$ is labeled as $\ten{A}_{[i_1i_4i_7][i_2i_5i_8][i_3i_6i_9]}$. Figure \ref{fig:d3kp}(a) illustrates how the grouped indices of the tensor $\ten{A}$ relate to those of the KP factors $\ten{A}^{(i)}_j$. Theorem \ref{theo:TKPSVD} tells us that the desired TKSPVD can be obtained from a PD of the permuted tensor $\ten{\tilde{A}}$. The first step in the TKPSVD algorithm is then to permute the indices of $\ten{A}$ such that their order corresponds with $i_1,i_2,i_3,i_4,i_5,i_6,i_7,i_8,i_9$. In order to do this, we first reshape the $3$-way tensor $\ten{A}$ into the $9$-way tensor $\ten{A}$ with entries $\ten{A}_{i_1i_4i_7i_2i_5i_8i_3i_6i_9}$. The indices of $\ten{A}$ are then permuted into the desired order $\tilde{\ten{A}}_{i_1i_2i_3i_4i_5i_6i_7i_8i_9}$. The next step of the TKPSVD algorithm is to compute the KP factors $\ten{A}^{(i)}_j$, each of which is computed as a vector in a PD. We therefore group the indices such that we obtain the $3$-way tensor $\tilde{\ten{A}}$ with entries $\tilde{\ten{A}}_{[i_1i_2i_3][i_4i_5i_6][i_7i_8i_9]}$. The steps prior to the computation of the PD are hence summarized as
\begin{align*}
 \ten{A}_{[i_1i_4i_7][i_2i_5i_8][i_3i_6i_9]} \xrightarrow{\textrm{reshape}} &\ten{A}_{i_1i_4i_7i_2i_5i_8i_3i_6i_9}  \xrightarrow{\textrm{permute}} \tilde{\ten{A}}_{i_1i_2i_3i_4i_5i_6i_7i_8i_9} &\\[1.5ex]
\xrightarrow{\textrm{reshape}}  \tilde{\ten{A}}_{[i_1i_2i_3][i_4i_5i_6][i_7i_8i_9]}.
\end{align*}
\begin{figure}[ht]
\centering
\includegraphics[width=.75\textwidth]{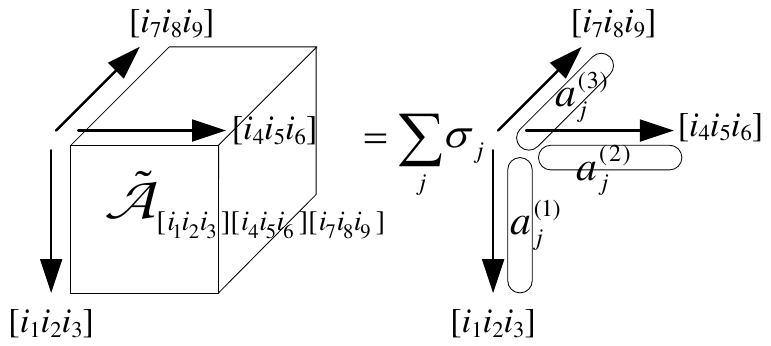}
\caption{Decomposition of the 3-way tensor $\tilde{\ten{A}}$ into a linear combination of rank-1 terms.}
\label{fig:d3ttr1}
\end{figure}
Each of the KP factors $\ten{A}^{(i)}_j$ is obtained from reshaping the $a^{(i)}_j$ vectors of the PD~(Figure \ref{fig:d3ttr1}) into a $3$-way tensor of the correct dimensions. In order to make sure this procedure of reshaping and applying the permutation is clear, we also demonstrate it for a simple matrix example. Suppose we have a $12 \times 12$ Hankel matrix $A$, which we want to decompose into a sum of KPs of a $4\times 4$ matrix with a $3\times 3$ matrix. If the entries of the KP factors $A^{(1)}_j,A^{(2)}_j$ are labeled by ${i_1i_2},{i_3i_4}$ respectively, then the row index of $A$ is $[i_1i_3]$ and the column index is $[i_2i_4]$, shown in Figure \ref{fig:d3kp}(b). The steps prior to the computation of the PD are now
\begin{align*}
 \ten{A}_{[i_1i_3][i_2i_4]} \xrightarrow{\textrm{reshape}} &\ten{A}_{i_1i_3i_2i_4}  \xrightarrow{\textrm{permute}} \tilde{\ten{A}}_{i_1i_2i_3i_4} \xrightarrow{\textrm{reshape}}  \tilde{\ten{A}}_{[i_1i_2][i_3i_4]}.
\end{align*}
The dimensions of the tensors in each of these steps are $12 \times 12, 4\times 3 \times 4 \times 3, 4 \times 4 \times 3 \times 3$ and $16 \times 9$ respectively. The $16 \times 9$ matrix $\tilde{A}$ shown on the front page is in fact obtained from this procedure. The Hankel structure is apparently lost due to the consecutive reshapings and permutation. The final step is to compute a PD with orthogonal rank-1 terms, which for the matrix case is the SVD. The previous two examples might give the impression that $d$ and $k$ need to be equal in the TKPSVD. This is not the case. In Figure \ref{fig:d3kp}(c) we show a degree-3 TKPSVD of a matrix. Here, the steps prior to the computation of the PD are
\begin{align*}
 \ten{A}_{[i_1i_3i_5][i_2i_4i_6]} \xrightarrow{\textrm{reshape}} &\ten{A}_{i_1i_3i_5i_2i_4i_6}  \xrightarrow{\textrm{permute}} \tilde{\ten{A}}_{i_1i_2i_3i_4i_5i_6} \xrightarrow{\textrm{reshape}}  \tilde{\ten{A}}_{[i_1i_2][i_3i_4][i_5i_6]}.
\end{align*}
The pseudo-code for our general TKPSVD algorithm is presented in Algorithm \ref{alg:tkpsvd}. As we will show in Section \ref{sec:strucproofs}, the structure-preserving property of the TKPSVD critically depends on the fact that the rank-1 terms of the computed PD are orthogonal with respect to one another. Another consequence of this orthogonality is that
$$
||\ten{A}||_F=\sqrt{\sigma_1^2 + \cdots +\sigma_R^2},
$$
where $R$ is the total number of terms in the decomposition. The relative approximation error obtained from truncating the number of terms to $r < R$ is then easily computed as
$$
\frac{||A-\sum_{j=1}^r \sigma_j\, A^{(d)}_j \otimes \cdots \otimes A^{(1)}_j ||_F } {||A||_F} = \frac{\sqrt{\sigma_{r+1}^2 + \cdots +\sigma_R^2}}{\sqrt{\sigma_1^2 + \cdots +\sigma_R^2}}.
$$
This is especially convenient when using the TTr1SVD~\cite{ttr1svd15} to compute the PD, since then all $\sigma_j$'s are positive and can be sorted in descending order just like singular values in the matrix case.\\
\\
\framebox[.99\textwidth][l]
{\begin{minipage}{0.99\textwidth}
\begin{algorithm}TKPSVD Algorithm\\
\textit{\textbf{Input}}: tensor $\ten{A}$, dimensions $n^{(1)}_1,\ldots,n^{(d)}_1,n^{(1)}_2,\ldots,n^{(d)}_2,\ldots,n_k^{(1)},\ldots,n^{(d)}_k$\\
\textit{\textbf{Output}}:\makebox[0pt][l]{ $\sigma_1,\ldots,\sigma_R$, tensors $\ten{A}^{(i)}_j$}
\begin{algorithmic}
\STATE $\ten{A} \gets$ reshape $\ten{A}$ into a ($kd$)-way tensor according to $n^{(1)}_1,\ldots,n^{(d)}_1,n^{(1)}_2,\ldots,n^{(d)}_k$
\STATE $\tilde{\ten{A}} \gets $ permute $\ten{A}$ to indices $i_1i_2i_3i_4\cdots i_{kd-1}i_{kd}$
\STATE \makebox[0pt][l]{$\tilde{\ten{A}} \gets $ reshape $\tilde{\ten{A}}$ into a $d$-way tensor by grouping every $k$ indices together}
\STATE $a^{(1)}_1,\ldots,a^{(d)}_R,\sigma_1,\ldots,\sigma_R \gets $ compute a PD of $\tilde{\ten{A}}$ with orthogonal rank-1 terms
\FOR {all nonzero $\sigma_j$}
\STATE $\ten{A}^{(i)}_j \gets $ reshape $a^{(i)}_j$ into a $n^{(i)}_1 \times \cdots \times n^{(i)}_k$ tensor
\ENDFOR
\end{algorithmic}
\label{alg:tkpsvd}
\end{algorithm}
\end{minipage}}\\
\\
The PD with orthogonal rank-1 terms is easily computed for the case $d=2$ as the SVD. When $d \geq 3$, several options are available. A first option is to compute the CPD with the additional constraint of orthogonal factor matrices. This orthogonality constraint limits the size of the  factor matrices and consequently also the total number of rank-1 terms that are possible to find. As a result, the CPD with orthogonality constraints does not lend itself very well to applications. We demonstrate this with a worked out example in Section \ref{sec:applications}. Alternatively, one could compute the HOSVD or the TTr1SVD of $\tilde{\ten{A}}$, as these decompositions have orthogonal rank-1 terms. The CPD with orthogonal factor matrices and the HOSVD can be computed in Matlab using Tensorlab~\cite{tensorlab}, freely available from~\url{http://www.tensorlab.net/}. A Matlab/Octave implementation of Algorithm \ref{alg:tkpsvd} that uses the TTr1SVD and works for any arbitrary degree $d$ and tensor order $k$ can be freely downloaded from \url{https://github.com/kbatseli/TKPSVD}.

In Section \ref{sec:gensym}, we introduce a new framework in which many different structured tensors (symmetric, persymmetric, centrosymmetric, Toeplitz, Hankel,...) can be described and then prove that the TKPSVD algorithm guarantees to preserve these structures in the KP factors. But first, we present a small modification of Algorithm \ref{alg:tkpsvd} for the case of diagonal tensors.

\subsection{Diagonal tensors}
A diagonal tensor $\ten{D}$ is an extremely simple symmetric tensor (see Section \ref{sec:gensym} for the definition). If we define the main diagonal of a cubical tensor as the entries $A_{i_1i_2\cdots i_k}$  with $i_1=i_2=\cdots=i_k$, then the entries not on the main diagonal of a diagonal tensor are per definition zero. It is easy to see that the Kronecker product of two diagonal tensors is also diagonal. This motivates us to adjust Algorithm \ref{alg:tkpsvd} such that only the main diagonal entries $\ten{D}_{i_1i_1\cdots i_1}$ are considered. This reduces the number of entries to store in memory from $n^d$ to $n$. As a result, the diagonal tensor $\ten{D}$ is decomposed into a Kronecker product of diagonal factors $\ten{D}^{(i)}_j$. Suppose a degree $d$ TKPSVD of a diagonal tensor $\ten{D}$ is required. We then consider the vector $a$ that contains all main diagonal entries with entries $a_{i_1i_2 \cdots i_d}$ and reshape it into a $d$-way tensor $\ten{A}$. Note that since the indices are already in the right order, no permutation of indices is required and the PD decomposition can be directly computed from $\ten{A}$. Each KP factor $\ten{D}^{(i)}_j$ of \eqref{eq:tkpsvd} is then an $n_i \times \cdots \times n_i$ diagonal tensor with main diagonal entries $a^{(i)}_j$. The pseudocode for the diagonal TKPSVD algorithm is summarized in Algorithm \ref{alg:dtkpsvd}.\\
\\
\framebox[.99\textwidth][l]{\begin{minipage}{0.99\textwidth}
\begin{algorithm}TKPSVD Algorithm for a diagonal tensor\\
\textit{\textbf{Input}}: diagonal tensor $\ten{D}$, dimensions $n_1,n_2,\ldots,n_d$\\
\textit{\textbf{Output}}:\makebox[0pt][l]{ $\sigma_1,\ldots,\sigma_R$, diagonal tensors $D^{(i)}_j$}
\begin{algorithmic}
\STATE $a \gets $ main diagonal entries of $\ten{D}$
\STATE $\ten{A} \gets$ reshape $a$ into an $n_1 \times n_2 \times \ldots \times n_d$ tensor
\STATE $a^{(1)}_1,\ldots,a^{(d)}_R,\sigma_1,\ldots,\sigma_R \gets $ compute a PD of ${\ten{A}}$ with orthogonal rank-1 terms
\FOR {all nonzero $\sigma_j$}
\STATE $\ten{D}^{(i)}_j \gets $ a $n_i \times \cdots \times n_i$ diagonal tensor with main diagonal entries $a^{(i)}_j$ 
\ENDFOR
\end{algorithmic}
\label{alg:dtkpsvd}
\end{algorithm}
\end{minipage}}\\
\\
It is interesting to investigate whether it is possible to adjust Algorithm \ref{alg:tkpsvd} to exploit other specific structures, like the general symmetries which we define in Section \ref{sec:gensym}. At first sight this does not seem to be straightforward to exploit, since $\tilde{\ten{A}}$ will not retain the original structure. We keep this problem for future research.

\begin{figure}[th]
\centering
\includegraphics[width=.65\textwidth]{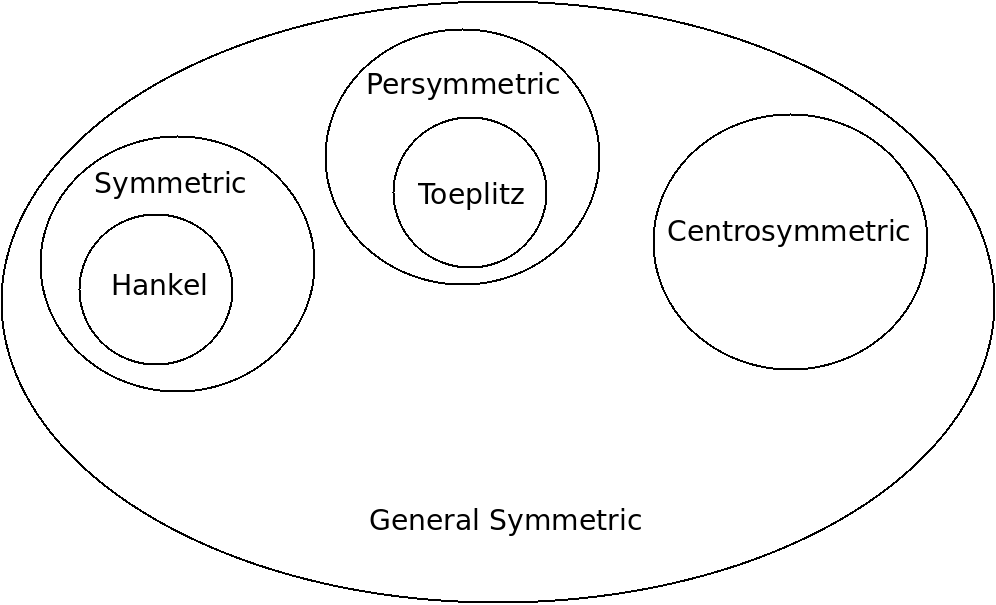}
\caption{Overview general symmetric tensors.}
\label{fig:gensym}
\end{figure}
\section{General symmetric tensors}
\label{sec:gensym}
It turns out that the $d$-way tensor $\tilde{\ten{A}}$ in Algorithm \ref{alg:tkpsvd} allows us to generalize the notion of symmetric tensors in a very natural way. The motivation of introducing general symmetry lies in the fact that then only one proof suffices to show the preservation of symmetry, persymmetry, centrosymmetry and many other symmetries in the KP factors of the TKPSVD. This new framework also provides a different perspective of describing and investigating these different tensor structures. Figure \ref{fig:gensym} shows an overview of how the notion of general symmetry encapsulates symmetric, persymmetric, centrosymmetric, Toeplitz and Hankel tensors. The key idea of the general symmetric structure is that it involves particular permutations $P$ of the entries of $\textrm{vec}(\ten{A})$ that can be decomposed into Kronecker product of smaller permutations along each mode of $\ten{\tilde{A}}$. We discuss and demonstrate this decomposition of the permutation $P$ for three particular cases. In the remainder of this Section we always assume that $\ten{A}$ is a $k$ way cubical tensor of dimensions $n$.

\subsection{Symmetry}
\label{subsec:symmetry}
The symmetric structure of a $k$-way cubical tensor $\ten{A}$ can be defined as a particular permutation of the entries of $\textrm{vec}(\ten{A})$. This permutation is described by the perfect shuffle matrix.
\begin{definition}
\label{def:genpershuf}
The perfect shuffle matrix $S$ is the $n^k \times n^k$ permutation matrix 
$$
S \;=\; \begin{pmatrix}
I(1:n^{k-1}:n^k,:)\\
I(2:n^{k-1}:n^k,:)\\
\vdots\\
I(n^{k-1}:n^{k-1}:n^k,:)\\
\end{pmatrix},
$$
where $I$ is the $n^k \times n^k$ identity matrix and Matlab colon notation is used to denote submatrices. 
\end{definition}

It is easily verified that for a symmetric $k$-way tensor $\ten{A}$ we have that $S\,\textrm{vec}(\ten{A})=\textrm{vec}(\ten{A})$. We now turn this reasoning on its head and define a symmetric tensor as any tensor $\ten{A}$ that satisfies $S\,\textrm{vec}(\ten{A})=\textrm{vec}(\ten{A})$. The perfect shuffle matrix $S$ reduces to the matrix defined in~\cite[p.~86]{Loan200085} for the case $k=2$. In this sense, Definition \ref{def:genpershuf} generalizes the notion of a perfect shuffle matrix to tensors.

In what follows we will apply the TKPSVD algorithm to construct the $\tilde{\ten{A}}$ tensor and see how this affects the equation $S\,\textrm{vec}(\ten{A})=\textrm{vec}(\ten{A})$. In order to illustrate this process we will consider the $3$-way example tensor from Section \ref{sec:tkpsvd} and suppose that it is symmetric. This symmetry implies that
\begin{equation}
\ten{A}_{[i_1i_4i_7][i_2i_5i_8][i_3i_6i_9]}=\ten{A}_{[i_2i_5i_8][i_3i_6i_9][i_1i_4i_7]}= \cdots = \ten{A}_{[i_3i_6i_9][i_1i_4i_7][i_2i_5i_8]}.
\label{eq:indexswap}
\end{equation}
In other words, the symmetry of $\ten{A}$ is equivalent with swapping $i_1$ with either $i_2$ or $i_3$, $i_4$ with either $i_5$ or $i_6$ and $i_7$ with either $i_8$ or $i_9$. The TKPSVD algorithm reshapes and permutes the symmetric tensor $\ten{A}$ into the tensor $\tilde{\ten{A}}$, with entries $\tilde{\ten{A}}_{[i_1i_2i_3][i_4i_5i_6][i_7i_8i_9]}$. Although $\tilde{\ten{A}}$ is not symmetric, the symmetry of $\ten{A}$ still allows us to swap the indices as indicated in \eqref{eq:indexswap} such that 
$$
\tilde{\ten{A}}_{[i_1i_2i_3][i_4i_5i_6][i_7i_8i_9]}=\tilde{\ten{A}}_{[i_4i_5i_6][i_7i_8i_9][i_1i_2i_3]}= \cdots = \tilde{\ten{A}}_{[i_7i_8i_9][i_1i_2i_3][i_4i_5i_6]},
$$
which can be rewritten as
\begin{equation}
\tilde{\ten{A}} \,\times_1 \, S_1 \,\times_2 \, S_2 \, \times_3 \, S_3 \;=\; \tilde{\ten{A}}
\label{eq:symmodes}
\end{equation}
where all $S_i$ matrices are perfect shuffle matrices. By using \eqref{eq:modeprodvec}, equation \eqref{eq:symmodes} can be rewritten as
\begin{equation}
(S_3  \otimes S_2 \otimes S_1) \, \textrm{vec}(\tilde{\ten{A}}) \;=\; \textrm{vec}(\tilde{\ten{A}}),
\label{eq:SvecA}
\end{equation}
which is nothing else but a reformulation of the symmetry $S \, \textrm{vec}(\ten{A})=\textrm{vec}(\ten{A})$ in terms of $\textrm{vec}(\tilde{\ten{A}})$. If $Q$ denotes the permutation matrix such that $Q\, \textrm{vec}(A) \;=\; \textrm{vec}(\tilde{\ten{A}})$, then from
\begin{align*}
(S_3\otimes  S_2 \otimes S_1) \, \textrm{vec}(\tilde{\ten{A}}) &= \textrm{vec}(\tilde{\ten{A}}), \\
\Leftrightarrow (S_3\otimes  S_2 \otimes S_1) \, Q\, \textrm{vec}(A)&= Q\, \textrm{vec}(A),\\
\Leftrightarrow  Q^T \, (S_3\otimes  S_2 \otimes S_1) \, Q\, \textrm{vec}(A)&=  \textrm{vec}(A),
\end{align*}
we infer that $S=Q^T \, (S_3\otimes  S_2 \otimes S_1)\,Q$. This can be interpreted as the perfect shuffle matrix $S$ being ``decomposed'' into a Kronecker product of smaller perfect shuffle matrices. Another way of seeing this equality is that $S$ and $S_3\otimes  S_2 \otimes S_1$ are permutation similar.

\subsection{Centrosymmetry}
Another interesting and useful permutation of the entries of $\textrm{vec}(\ten{A})$ is the exchange matrix $J$, which is the $n^k \times n^k$ column-reversed identity matrix. This permutation maps each index $i_j$ of $\textrm{vec}(\ten{A})$ to $n-i_j+1$, e.g., for the $3$-way tensor $\ten{A}$ from Section \ref{sec:tkpsvd} the entry $\textrm{vec}(\ten{A})_{[i_1i_4i_7i_2i_5i_8i_3i_6i_9]}$ is mapped to
\begin{align*}
\textrm{vec}(\ten{A})_{[n-i_1+1\; n-i_4+1\; n-i_7+1\; n-i_2+1\; n-i_5+1\; n-i_8+1\; n-i_3+1\; n-i_6+1\; n-i_9+1]}
\end{align*}
and vice-versa.
A $k$-way cubical tensor $\ten{A}$ is defined to be centrosymmetric when
$$
J \, \textrm{vec}(\ten{A}) \;= \; \textrm{vec}(\ten{A}).
$$
Our definition of centrosymmetric tensors is equivalent with the alternative definitions given in~\cite{chenchenqi2014,zhaoyang2014}. The ``decomposition" argument of $J$ is completely analogous to the decomposition of the perfect shuffle matrix $S$ for symmetric tensors. Following the TKPSVD reshapings and permutations leads to the expression $J\;=\; Q^T \, (J_3 \otimes  J_2 \otimes J_1)\,Q$, where $J_1,J_2,J_3$ are exchange matrices and $Q$ is the same permutation as described in Subsection \ref{subsec:symmetry}.

\subsection{Persymmetry}
Given the definition of the perfect shuffle matrix $S$ and exchange matrix $J$, we now define a $k$-way cubical tensor $\ten{A}$ to be persymmetric when
$$
S\, J \, \textrm{vec}(\ten{A}) \;=\; \textrm{vec}(\ten{A}),
$$
applies. Using similar arguments as in the symmetric and centrosymmetric cases, we can write the following decomposition $S\, J\;=\; Q^T \, (S_3\,J_3  \otimes S_2\,J_2 \otimes S_1\,J_1)\,Q$. Using the mixed-product property of the Kronecker product, we can rewrite the permutation decomposition as $S\, J\;=\; Q^T \, (S_3\otimes S_2 \otimes S_1) (J_3  \otimes J_2 \otimes J_1)\,Q$.

\subsection{General symmetric tensor}
We now define general symmetric tensors by generalizing the previous three examples of particular symmetries. 
\begin{definition}
A $k$-way cubical tensor $\ten{A}$ is a general symmetric tensor if
$$
P\, \textrm{vec}(\ten{A}) \;=\; \textrm{vec}(\ten{A}),
$$
where the permutation matrix $P$ can be written for any arbitrary degree $d$ and dimensions $n^{(i)}_r$ into a Kronecker product of smaller permutation matrices $P_1,\ldots,P_d$ as
\begin{equation}
P \;=\; Q^T \, (P_d \otimes \cdots \otimes  P_2 \otimes P_1)\,Q,
\label{eq:permcond}
\end{equation}
where $Q$ is the permutation matrix such that $Q\, \textrm{vec}(\ten{A})=\textrm{vec}(\tilde{\ten{A}})$ and the dimension of each of the permutation matrices $P_i$ is $\prod_{r=1}^k n^{(i)}_r$.
\label{def:gensym}
\end{definition}

General skew-symmetric tensors are defined similarly as in Definition \ref{def:gensym} where now $P\, \textrm{vec}(\ten{A}) = - \textrm{vec}(\ten{A})$ needs to hold. The permutation matrices $P_1,\ldots,P_k$ do not necessarily need to be of the same dimension. In fact, the definition requires that \eqref{eq:permcond} holds for any arbitrary degree $d$ and dimensions $n^{(i)}_r$ of the KP factors. For example, there are several ways in which a general symmetric $12 \times 12 \times 12$ tensor $\ten{A}$ can be decomposed into a TKPSVD. There are three different decompositions when $d=3$, depending on the order of the $3 \times 3 \times 3, 2\times 2 \times 2$ and $2\times 2 \times 2$ tensors. Likewise when $d=2$ there are different orderings of the $3 \times 3 \times 3, 4 \times 4 \times 4$ or $6 \times 6 \times 6, 2 \times 2 \times 2$ tensors. Each of these TKPSVDs is characterized by different $P_k$ and $Q$ permutation matrices, nevertheless, \eqref{eq:permcond} needs to hold for all of them for $\ten{A}$ to be general symmetric. 

\subsection{Shifted-index structure}
Within the set of general symmetric tensors there are other interesting, more restrictive structures that we call shifted-index structures. These are tensors whose entries do not change when at least one index is ``shifted''.
\begin{definition}
A $k$-way cubical tensor $\ten{A}$ has a shifted-index structure if
$$
A_{[i_1][i_2]\cdots [i_k]} \;=\; A_{[i_1+\Delta_1][i_2+\Delta_2]\cdots [i_k+\Delta_k]},
$$
where at least one of the integer shifts $\Delta_1,\ldots,\Delta_k$ is nonzero. For any two nonzero shifts $\Delta_i,\Delta_j$, either $\Delta_i=\Delta_j$ or $\Delta_i=-\Delta_j$ must be satisfied.
\end{definition}

Of course, none of the shifted indices $i_1+\Delta_1,\ldots,i_k+\Delta_k$ are allowed to go ``out of bounds''. The case where $\Delta_1=\Delta_2=\cdots = \Delta_k$ is called a Toeplitz tensor and is a special case of a persymmetric tensor. Similarly, a symmetric tensor for which all shifts are zero except for one arbitrary pair $\Delta_i=-\Delta_j$ is called a Hankel tensor. A tensor for which all shifts are zero except $\Delta_r$ has constant entries along the $r$ fibres. It is straightforward to show that any shifted-index structure is also a general symmetry by writing down its corresponding permutation matrix and showing that Definition \ref{def:gensym} applies.
\begin{example}
Consider a $27 \times 27 \times 27$ Hankel tensor $\ten{H}$. The Hankel structure means that the tensor is also symmetric, which implies that $S\,\textrm{vec}(\ten{H}) = \textrm{vec}(\ten{H})$ with $S$ the $19683 \times 19683$ perfect shuffle matrix. Now consider a degree-3 TKPSVD where each KP factor is a $3 \times 3 \times 3$ tensor. We can then retrieve $S$ from the $27 \times 27$ perfect shuffle matrix $S_1=S_2=S_3$ as $S = Q^T\,(S_3 \otimes S_2 \otimes S_1)\,Q$, where $Q$ is the permutation matrix in $\textrm{vec}(\tilde{\ten{H}})=Q\,\textrm{vec}(\ten{H})$. The $27 \times 27$ permutation matrices $P_1=P_2=P_3$ that define a $3 \times 3 \times 3$ Hankel tensor $\ten{A}$ are completely specified by the vector of indices
$$
i= [1,4,5,10,7,8,11,12,15,2,13,14,19,16,17,20,21,24,3,22,23,6,25,26,9,18,27],
$$
since $\textrm{vec}(\ten{A})(i)=\textrm{vec}(\ten{A})$, where $\textrm{vec}(\ten{A})(i)$ is Matlab notation to denote $P_3\,\textrm{vec}(\ten{A})$. If we now set $P=Q^T\,(P_3 \otimes P_2 \otimes P_1)\,Q$ then indeed $P\,\textrm{vec}(\ten{H})=\textrm{vec}(\ten{H})$ is satisfied. Figure \ref{fig:perms} shows the nonzero pattern of both $S$ and $P$. Observe that although the Hankel permutation $P$ is a special case of a symmetry, the nonzero pattern is very different from that of $S$.
\end{example}

\begin{figure}
    \centering
    \begin{subfigure}[b]{0.48\textwidth}
        \includegraphics[width=\textwidth]{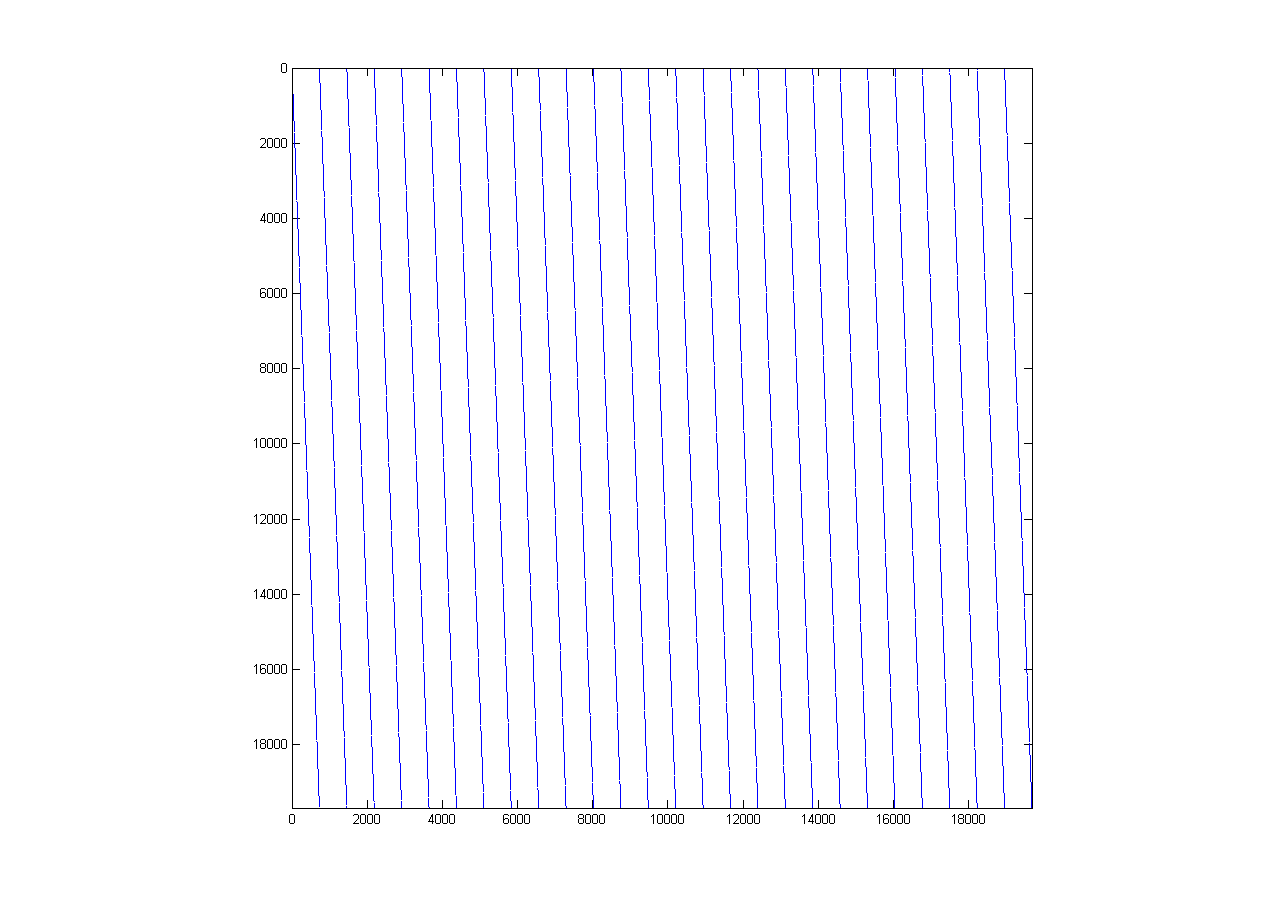}
        \caption{Perfect shuffle matrix $S$}
        \label{fig:psym}
    \end{subfigure}
    ~ 
    \begin{subfigure}[b]{0.48\textwidth}
        \includegraphics[width=\textwidth]{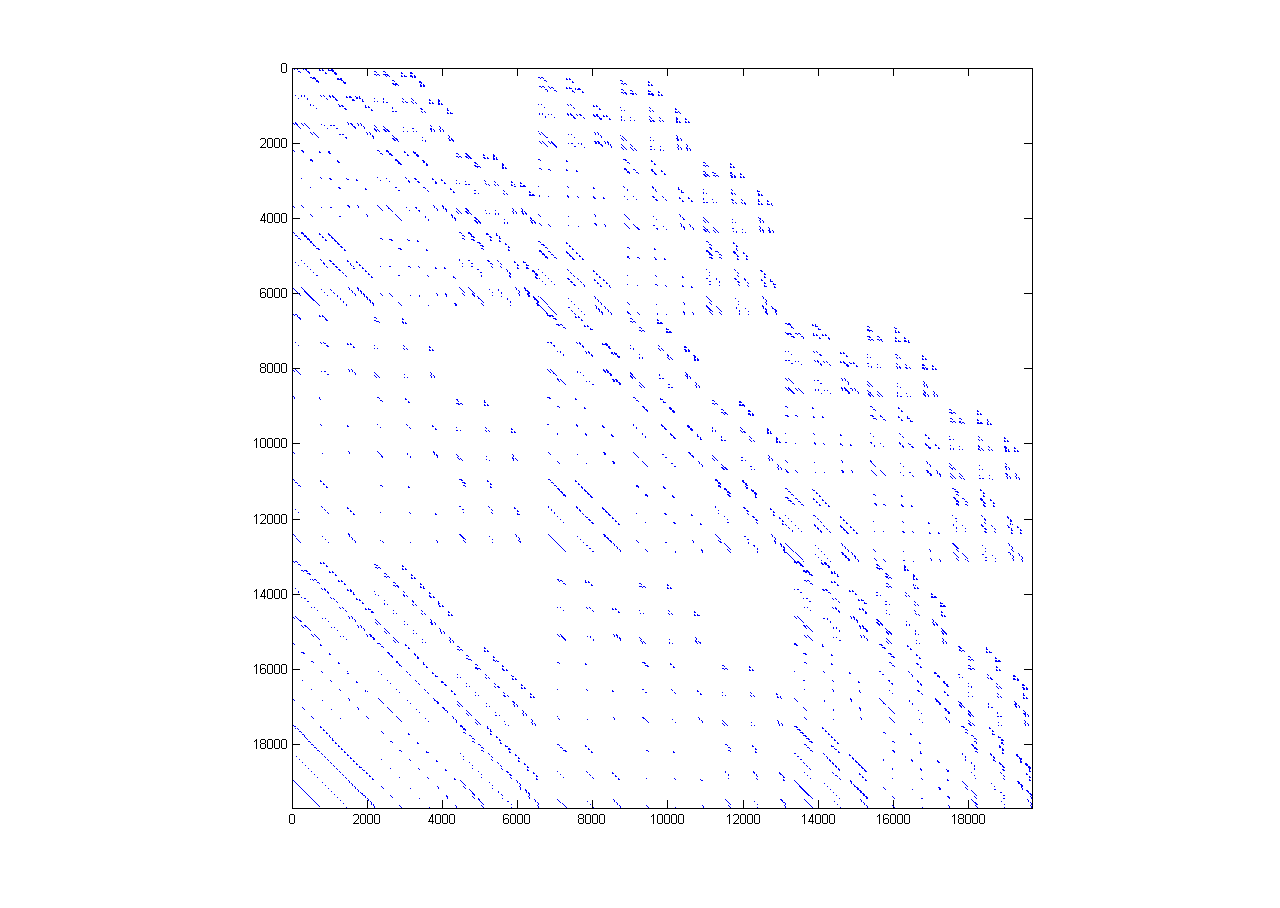}
        \caption{Hankel permutation matrix $P$}
        \label{fig:phankel}
    \end{subfigure}
    ~ 
   \caption{Permutation matrices for a $27 \times 27 \times 27$ Hankel tensor.}\label{fig:perms}
\end{figure}

\section{Preservation of structures}
\label{sec:strucproofs}
It is quite a remarkable fact that all general symmetries, included the shifted-index structures, are preserved in the cubical KP factors $\ten{A}^{(i)}_j$ when they are computed according to Algorithm \ref{alg:tkpsvd}. The orthogonality of the rank-1 terms in the PD plays a crucial role in this. Another critical element are the scalar coefficients $\sigma_j$'s of the TKPSVD, which are required to be distinct. We now show how this comes about.
\subsection{General symmetry}
In order to prove general symmetry preservation in the KP factors, we need the following useful lemma.
\begin{lemma}
Suppose $a=\textrm{vec}(\ten{A}) \in \mathbb{R}^{n^k \times 1}$ with $a^Ta=1$ and $P$ is a permutation matrix that corresponds with a general symmetry. Then the cubical tensor $\ten{A}$ obtained from reshaping $a$ is general symmetric if and only if $a^T\,P\,a=1$ or general skew-symmetric if and only if $a^T\,P\,a=-1$.
\label{lemma:symA}
\end{lemma}
\begin{proof}
We first prove $Pa=a \Rightarrow a^T\,P \, a=1$. Since $a$ has unit norm we can write $a^Ta=1$ and substitution of $a$ by $Pa$ then results in $a^TPa=1$. The proof for $a^T\,P \, a=1 \Rightarrow Pa=a$ goes as follows. Let $b=P\,a$, then we have that $||b||_2=1$ and $a^T\,b = \cos{\alpha}$. Since $\cos{\alpha}=1$ it follows that $\alpha = 0$ and $a$ is a multiple of $b$, but since $||a||_2=||b||_2=1$ it follows that $a=b=P\,a$.
The proof for the skew-symmetry of $\ten{A}$ follows the same logic.
\quad \end{proof}

The general symmetry of $\ten{A}$ implies that
\begin{equation}
\tilde{\ten{A}}=\tilde{\ten{A}} \times_1 P_1 \times_2 P_2 \times_3 \cdots \times_d  P_d,
\label{eq:symAtilde}
\end{equation}
where all $P_i$'s are permutation matrices. We now substitute $\tilde{\ten{A}}$ in both sides of \eqref{eq:symAtilde} by its PD and obtain
\begin{equation}
\sum_{j=1}^R \sigma_j \, a^{(1)}_j \circ a^{(2)}_j \circ \cdots \circ a^{(d)}_j \;=\; \sum_{i=1}^R \sigma_j \, P_1a^{(1)}_j \circ P_2a^{(2)}_j \circ \cdots \circ P_da^{(d)}_j.
\label{eq:symAtilde2}
\end{equation}
The orthogonality of each rank-1 term and $||a^{(i)}_j||_2=1$ implies that the mode products of both sides of \eqref{eq:symAtilde2} with $(a^{(1)}_k)^T,\ldots,(a^{(d)}_k)^T$ along the modes $1,2,\ldots,d$ respectively for any $k \in \{1,\ldots,R\} \subset \mathbb{N}$ results in
\begin{equation}
\sigma_k \;=\; \sum_{j=1}^R \sigma_j \; (a^{(1)}_k)^T P_1\, a^{(1)}_j \; (a^{(2)}_k)^T P_2\, a^{(2)}_j \, \cdots \, (a^{(d)}_k)^T P_d\, a^{(d)}_j.
\label{eq:symAtilde3}
\end{equation}
We have that each of the $(a^{(i)}_k)^T P_i a^{(i)}_j$ scalars lies in the real interval $[-1,1]$, since $||a^{(i)}_j||_2=1$ for all $j$ and $P$ is a permutation matrix. We now assume that the following condition holds.
\begin{condition}
\label{cond:distinctsigma}
All $\sigma_j$'s in the PD of $\tilde{\ten{A}}$ are distinct and all terms on the right-hand side of \eqref{eq:symAtilde3} except for the one corresponding with $\sigma_k$ vanish.
\end{condition}

The equality in \eqref{eq:symAtilde3} holds under Condition \ref{cond:distinctsigma} when $(a^{(i)}_k)^T\,P_i\,a^{(i)}_j=0$ for at least one particular $i$ when $k \neq j$ and when
\begin{equation}
\prod_{i=1}^{d} (a^{(i)}_k)^T\,P_i\,a^{(i)}_k \;=\; 1.
\label{eq:equalcond}
\end{equation}

The only possible way for \eqref{eq:equalcond} to be true under the constraint that each of the $(a^{(i)}_k)^T P_i a^{(i)}_k$ scalars lies in the interval $[-1,1]$ is when 
\begin{equation}
(a^{(i)}_k)^T\,P_i\,a^{(i)}_k \;=\; \pm 1.
\label{eq:symorskew}
\end{equation}
From Lemma \ref{lemma:symA} we know that if the right-hand side of \eqref{eq:symorskew} is 1, then $\ten{A}^{(i)}_k$ is general symmetric, otherwise $\ten{A}^{(i)}_k$ is general skew-symmetric. In addtion, \eqref{eq:equalcond} implies that there are either zero or an even number of general skew-symmetric KP factors in each term. This proves the main theorem on general symmetry-preservation in the TKPSVD.

\begin{theorem}
Let $\ten{A}$ be a general symmetric tensor with a $d$th-degree TKPSVD into cubical KP factors $A^{(i)}_j$. If Condition \ref{cond:distinctsigma} holds, then each of the $A^{(i)}_j$ factors in the TKPSVD is either a general symmetric or a general skew-symmetric tensor. There are always either zero or an even number of skew-symmetric factors in each term of \eqref{eq:tkpsvd}.
\label{theo:symA}
\end{theorem}

Let us see what happens when Condition \ref{cond:distinctsigma} is not satisfied. Suppose that $\sigma_k=\sigma_{k+1}$ and that the terms corresponding with the other $\sigma_j$'s vanish. We can then combine the $\sigma_k$ and $\sigma_{k+1}$ terms such that now 
\begin{equation}
\prod_{i=1}^{d} (a^{(i)}_k)^T\,P_i\,a^{(i)}_k +  \prod_{i=1}^{d} (a^{(i)}_{k+1})^T\,P_i\,a^{(i)}_{k+1}\;=\; 1
\label{eq:equalcond2}
\end{equation}
needs to hold. Contrary to the case of distinct $\sigma_j$'s, there are now multiple ways that \eqref{eq:equalcond2} can be satisfied without $(a^{(i)}_k)^T\,P_i\,a^{(i)}_k = \pm 1$ being true, which implies that the terms corresponding with $\sigma_k,\sigma_{k+1}$ will not necessarily have general symmetric factors. A very particular case where Condition \ref{cond:distinctsigma} is not satisfied for all terms is for symmetric tensors of order $k > 2$. We discuss this case, along with other general symmetries, in Section \ref{sec:applications}.

\section{Numerical Experiments}
\label{sec:applications}
In this section we discuss numerical experiments that illustrate different aspects of the TKPSVD algorithm. We will demonstrate the structure preservation of generaly symmetries and shifted-index structures and discuss a curious observation for symmetric tensors. We also compare the use of the CPD with orthogonal matrix factors, the HOSVD and TTr1SVD in terms of runtime and storage. Finally, we illustrate how the Kronecker product structure can be interpreted as a multiresolution decomposition of images. All computations were done in Matlab on a 64-bit 4-core 3.3 GHz desktop computer with 16 GB RAM.

\subsection{General symmetric structure}
As a first example, we demonstrate the use of the CPD with orthogonal factor matrices, the HOSVD and the TTr1SVD to compute the TKPSVD, together with the preservation of centrosymmetry in the KP factors.
\begin{example}
\label{ex:ex1}
We construct a $24 \times 24  \times 24$ centrosymmetric tensor $\ten{A}$ with its distinct entries drawn from a standard normal distribution and compute a 3rd degree decomposition with factor sizes $4\times 4 \times 4, 3\times 3 \times 3, 2 \times 2 \times 2$ respectively. The reshaping and permutation steps result in a $8 \times 27 \times 64$ tensor $\tilde{\ten{A}}$. Table \ref{tab:ex1} compares the use of the CPD with orthogonal factor matrices, the HOSVD and TTr1SVD for the PD of $\tilde{\ten{A}}$. We list the total number of rank-1 terms, the required memory for storage of the PD of $\tilde{\ten{A}}$, the total runtime to compute the TKPSVD (computed as the median over 100 runs) and the relative error $ {||\ten{A}-\sum_{j=1}^R \sigma_j  \ten{A}^{(d)}_j \cdots \ten{A}^{(1)}_j ||_F}/{||\ten{A}||_F}$. The total number of rank-1 terms for the orthogonal CPD is limited to only 8, since the first factor matrix will be an orthogonal $8 \times 8$ matrix. This results in a large relative error. The runtime for computing the orthogonal CPD is also very long compared to the HOSVD and TTr1SVD. The main difference between the HOSVD and TTr1SVD lies in the total number of rank-1 terms. But since the HOSVD reuses the mode vectors, this results in slightly less required memory. All rank-1 terms computed with both the HOSVD and TTr1SVD retain the centroysymmetric structure. The TTr1SVD method results in 56 terms that have 2 skew-centrosymmetric factors, The HOSVD has 1792 such terms. Note that the HOSVD has a $8 \times 27 \times 64$ core tensor containing 6912 nonzero entries.

\begin{table}[ht]
\centering
\caption{Comparison of CPD, HOSVD and TTr1SVD for the TKPSVD.}
\label{tab:ex1} 
\begin{tabular}{@{}lrrrr@{}}
Method & \# Terms & Storage (kB) & Runtime (seconds) & Relative Error\\ \midrule
orthogonal CPD  & 8 & 3.14 & 768.07 & 0.947 \\
HOSVD & 6912 & 146.19 & 0.18 & 2.21e-15 \\
TTr1SVD & 216 & 154.06 & 0.21 & 2.39e-15 \\
\end{tabular}
\end{table}
\end{example}
Due to the limitations of the CPD with orthogonal factors, as demonstrated in Example \ref{ex:ex1}, we will refrain from using it for the following examples in this Section. We now demonstrate the occurrence of $\sigma$'s with multiplicities for symmetric tensors. Consequently, KP terms that belong to the same $\sigma$ will not inherit the general symmetry in their factors. This is true for when both the TTr1SVD and HOSVD are used. The TTr1SVD case however has a lot more regularity than the HOSVD. When the TKPSVD of a $k$-way symmetric tensor is computed with the TTr1SVD, then all multiple $\sigma$'s have a multiplicity of $k-1$. 
\begin{example}
Consider a symmetric $8 \times 8 \times 8$ tensor with distinct entries drawn from a standard normal distribution. We compute its degree-3 TKPSVD using the TTr1SVD and obtain 56 KP terms. For this $3$-way tensor, each multiple $\sigma$ has a multiplicity of $k-1=2$. There are 8 such pairs, which implies that 16 terms are not (skew)-symmetric. Next, we compute a degree-3 decomposition for an $8 \times 8$ symmetric matrix, which results in 14 KP terms. All the $\sigma$'s are distinct, which implies that all terms in the decomposition are (skew)-symmetric. Finally, we compute the degree-3 TKPSVD of a 4-way symmetric tensor. This decomposition consists of 230 terms. There are 20 3-tuples of multiple $\sigma$'s, which means that $60$ KP terms are not (skew)-symmetric.
\end{example}

\subsection{Shifted-index structure}
Next, we investigate the dependence of the total number of KP terms and total runtime on the ordering of the KP factors for both the TTr1SVD and HOSVD.
\begin{example}
Consider a $64 \times 64 \times 64 \times 64$ Hankel tensor with its distinct entries drawn from a standard normal distribution. We compute its TKPSVD into 3 Kronecker factors with dimensions $2 \times 2 \times 2 \times 2, 4 \times 4 \times 4 \times 4$ and $8 \times 8 \times 8 \times 8$ over all possible orderings of the factors and investigate the total number of obtained KP terms for both the TTr1SVD and HOSVD, together with total runtimes. The results are shown in Table \ref{tab:ex3}. The first thing to notice is that the total number of terms in the TKPSVD and total runtime are quite independent from the factor ordering when the HOSVD is used. On average, about 2000 terms are needed and the computation takes a little over 4 minutes. The TTr1SVD needs about 20 times less terms and is for one particular ordering more than 80 times faster. Note that although the HOSVD requires more terms, just like in Example \ref{ex:ex1} it will require less memory for storage due to the fact that it reuses the mode vectors for each KP term. All of the terms in every of the decompositions retained the Hankel structure.
\begin{table}[ht]
\begin{center}
\caption{Number of KP terms and total runtime for TTr1SVD and HOSVD.}
\label{tab:ex3}	
\begin{tabular}{@{}lcccc@{}}
		Ordering & \multicolumn{2}{c}{\# Terms} & \multicolumn{2}{c}{Runtime (seconds)} \\
\midrule
      & TTr1SVD & HOSVD & TTr1SVD & HOSVD \\ \midrule
2,4,8 & 65 & 1968 & 9.76  & 254.30 \\
2,8,4 & 65 & 1973 & 2.47  & 251.70\\
4,2,8 & 65 & 1993 & 5.80  & 254.96\\
4,8,2 & 65 & 2146 & 5.75  & 256.41\\
8,2,4 & 145& 2067 & 247.66& 249.97\\
8,4,2 & 145& 2105 & 239.95& 254.16\\
\end{tabular}
\end{center}
\end{table}

\end{example}

\subsection{Multiresolution decomposition of images}
\begin{figure}[ht]
\centering
\includegraphics[width=.75\textwidth]{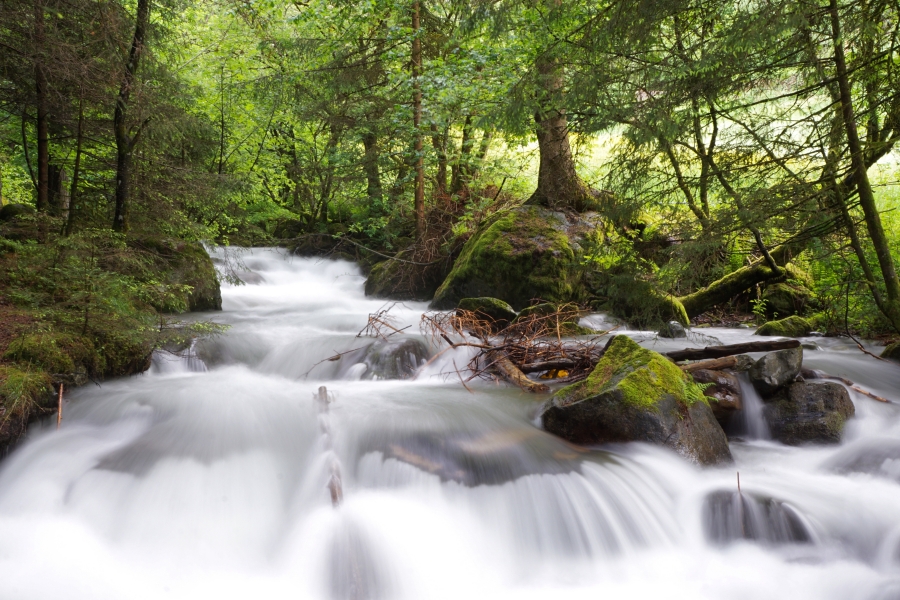}
\caption{Original $4000 \times 6000 \times 3$ image.}
\label{fig:mountainlake}
\end{figure}
An interesting illustration of the TKPSVD is in the multiresolution decomposition of a $n_1 \times n_2 \times 3$ colour image $\ten{A}$. This example gives an interpretation to two different aspects of the TKPSVD: truncation of the Kronecker product and truncation of the number of terms. Indeed, every pixel of the $n_1^{(d)} \times n_2^{(d)} \times 3$ image $\ten{A}^{(d)}$ is ``blown up'' by each Kronecker product in $\ten{A}^{(d)} \otimes \ten{A}^{(d-1)} \otimes \cdots \otimes \ten{A}^{(1)}$ until the resolution $n_1 \times n_2 \times 3$ is obtained. Truncating the Kronecker products to only a few factors hence effectively reduces the resolution. Furthermore, compression can be achieved at different resolutions by retaining only a few terms. The compression rate achieved for retaining $k$ factors and $r$ number of terms in \eqref{eq:tkpsvd} is defined as
$$
\frac{\prod_{i=0}^{k-1} n^{(d-i)}_1 \, n_2^{(d-i)} \,n_3^{(d-1)}}{r \sum_{i=0}^{k-1}  n^{(d-i)}_1 \, n_2^{(d-i)} \,n_3^{(d-1)} }.
$$
This is illustrated with the $6000\times 4000 \times 3$ colour image in Figure~\ref{fig:mountainlake}\footnote{\url{absfreepic.com/free-photos/download/water-nature-fall-6000x4000_90673.html}}. A degree-5 TKPSVD is computed with dimensions
$$(250 \times 375 \times 3)\otimes (2 \times 2 \times 1)\otimes (2 \times 2 \times 1) \otimes (2\times 2 \times 1) \otimes (2 \times 2 \times 1).
$$
The total runtime to compute the TKPSVD using the TTr1SVD and HOSVD was 18 and 30 seconds respectively.  In contrast, computing a standard SVD of one of the three slices $\ten{A}(:,:,i)$ takes about 1 minute and consists of 4000 rank-1 terms. The TTr1SVD needs 240 terms while the HOSVD needs 65536. For this particular example, the compression rate when retaining $k$ factors and $r$ terms can be approximated by
$$
\frac{(250\times 375 \times 3)\,(2\times 2 \times 1)^{k-1} }{r (250\times 375 \times 3+(k-1)( 2\times 2 \times 1)) } \approx \frac{(2\times 2 \times 1)^{k-1}}{r} = \frac{4^{k-1}}{r}.
$$
This implies that the maximal compression rate, when $r=1$, is dependent on the resolution, viz. the number of KP factors $k$ in each term. At the largest resolution $(k=5)$, the maximal compression rate is approximately $256$ while at the smallest resolution no compression is possible through truncation of KP terms. A common measure to quantify the quality of reconstruction of lossy compressed images is the peak signal-to-noise ratio (PSNR). The PSNR is defined as
$$
PSNR = 20 \log_{10}(\textrm{MAX}_\textrm{I}) - 10 \log_{10}(\textrm{MSE})
$$
where $\textrm{MAX}_{\textrm{I}}$ is the maximal possible pixel value, 255 in our case, and MSE is the mean squared error $||\ten{A}-\hat{\ten{A}}||^2_F / (n_1\cdot n_2 \cdot 3)$. Figure \ref{fig:PSNR} shows the PSNR as a function of the number of retained KP terms, computed from the TTr1SVD for the highest possible resolution. In this case, the PSNR can be completely determined from the $\sigma$'s in the TKPSVD using \eqref{eq:approxerror}. Acceptable values of the PSNR are between 30 and 50 dB and are obtained from retaining the first 20 terms, which corresponds with a compression rate of about $12.8$. Figure \ref{fig:multiscale} displays 1-term approximants for 3 different resolutions. For Figures \ref{fig:multiscale}(a),(b),(c) the PSNR is $58 \textrm{dB}, 58 \textrm{dB}$ and $57 \textrm{dB}$ respectively.

\begin{figure}[ht]
\centering
\includegraphics[width=.75\textwidth]{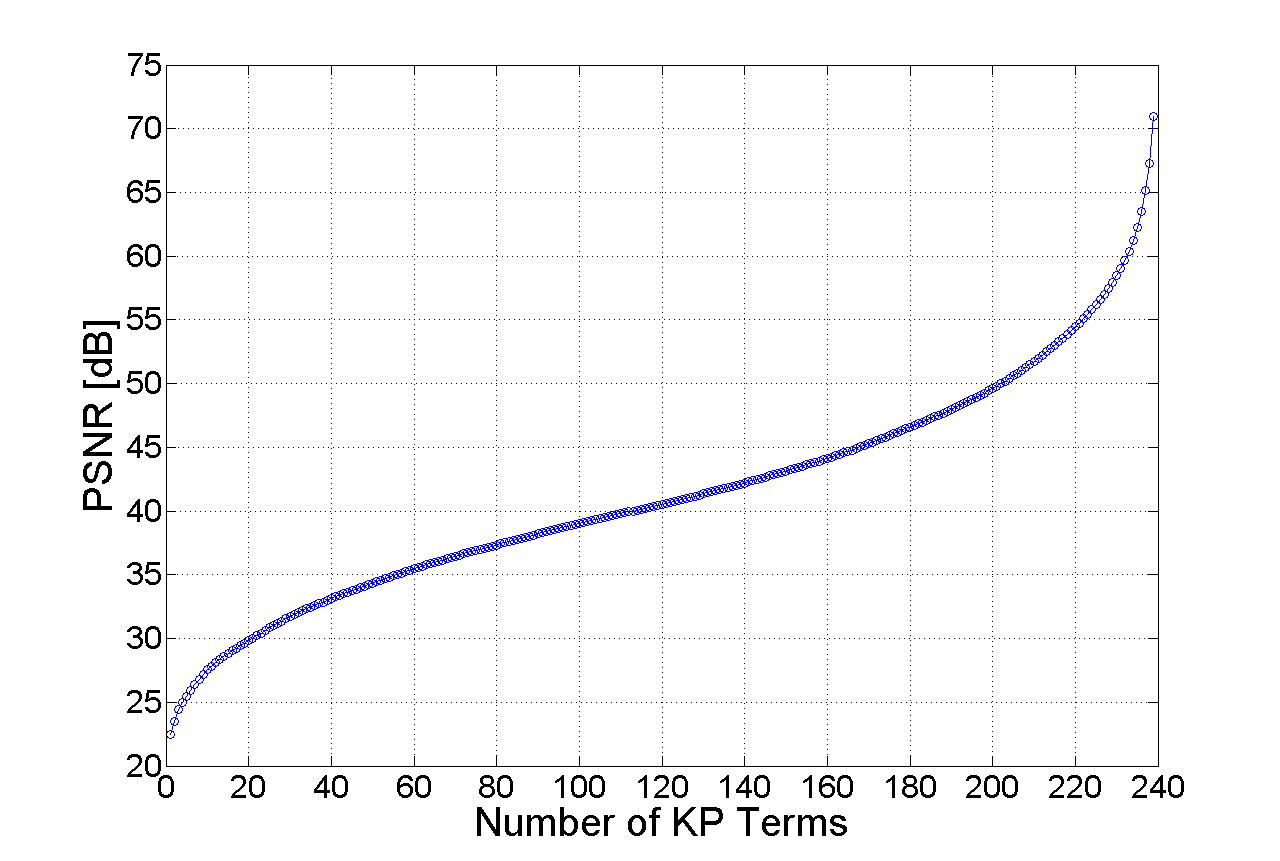}
\caption{Increase of the PSNR as the number of TTr1SVD-KP terms grows for the $4000\times 6000 \times 3$ resolution.}
\label{fig:PSNR}
\end{figure}

\begin{figure}[th]
    \centering
    \begin{subfigure}[b]{0.2\textwidth}
        \centering
        \includegraphics[width=\textwidth]{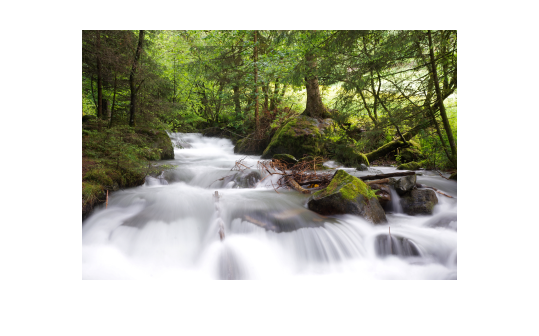}
        \caption{$250 \times 375$}
        \label{fig:100x150}
    \end{subfigure}
    \begin{subfigure}[b]{0.4\textwidth}
        \centering
        \includegraphics[width=\textwidth]{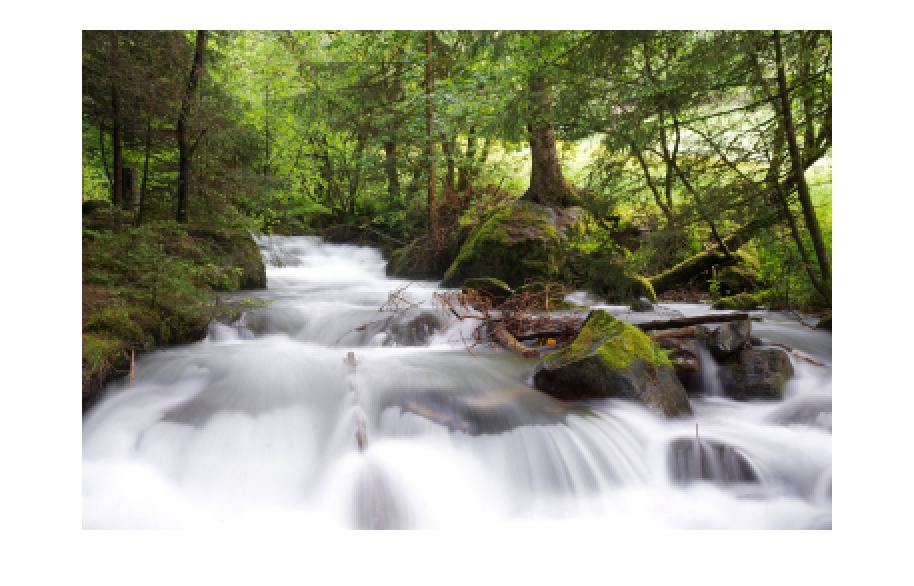}
        \caption{$500 \times 750$}
        \label{fig:400x600}
    \end{subfigure}
    \begin{subfigure}[b]{0.8\textwidth}
        \centering
        \includegraphics[width=\textwidth]{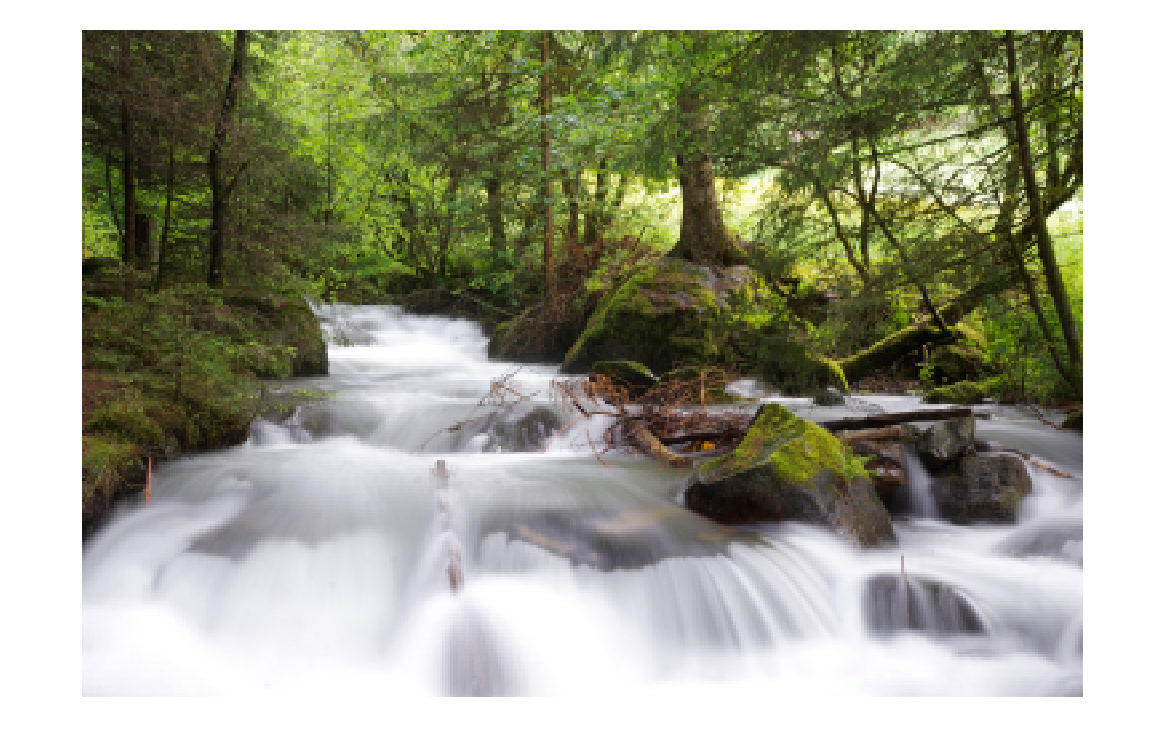}
        \caption{$1000 \times 1500$}
        \label{fig:800x1200}
    \end{subfigure}
    \caption{First term of the KP decomposition for 3 different resolutions.}
    \label{fig:multiscale}
\end{figure}

\section{Conclusions}
\label{sec:conclusions}
In this paper, we introduced the tensor Kronecker product singular value decomposition that decomposes a real $k$-way tensor $\ten{A}$ into a linear combination of tensor Kronecker product terms with an arbitrary number of $d$ factors $\ten{A} = \sum_{j=1}^R \sigma_j\, \ten{A}^{(d)}_j \otimes \cdots \otimes \ten{A}^{(1)}_j$. This decomposition enables easy computation of a Kronecker product approximation and a very straightforward determination of the relative approximation error without explicit construction of the approximant. We proved that for many different structured tensors, the Kronecker product factors $\ten{A}^{(1)}_j,\ldots,\ten{A}^{(d)}_j$ are guaranteed to inherit this structure. In addition, we introduced the new framework of general symmetric tensors, which includes many different structures such as symmetric, persymmetric, centrosymmetric, Toeplitz and Hankel tensors. 

\section*{Acknowledgements}
The authors would like to thank Martijn Bouss\'e and Nico Vervliet for their invaluable help on computing a CPD with orthogonal factor matrices in Tensorlab.
 
 \bibliographystyle{siam}
 \bibliography{references.bib}

\end{document}